\documentclass[a4paper,12pt,leqno]{article}

\usepackage[matrix,arrow,tips,curve]{xy}
\usepackage[english]{babel}
\usepackage{amsmath}
\usepackage{amssymb,amsthm}
\usepackage{enumerate}

\newtheorem{thm}{Theorem}
\newtheorem*{thm*}{Theorem}
\newtheorem{lemma}[thm]{Lemma}
\newtheorem{proposition}[thm]{Proposition}

\newtheorem{conjecture}[thm]{Conjecture}
\newtheorem{step}{Step}

\theoremstyle{definition}

\newtheorem{remark}[thm]{Remark}

\newcommand{\ph}{\varphi}
\newcommand{\w}{\widetilde}
\newcommand{\pr}{\mathbb{P}}
\newcommand{\Q}{\mathbb{Q}}
\newcommand{\R}{\mathbb{R}}
\newcommand{\N}{\mathcal{N}_1}
\newcommand{\Nu}{\mathcal{N}^1}

\newcommand{\NE}{\operatorname{NE}}
\newcommand{\Exc}{\operatorname{Exc}}
\newcommand{\Lo}{\operatorname{Locus}}
\newcommand{\codim}{\operatorname{codim}}
\newcommand{\dom}{\operatorname{dom}}

\newlength{\Mheight}
\newlength{\cwidth}
\title{Numerical invariants of Fano 4-folds}
\author{Cinzia Casagrande}
\date{January 25, 2012}
\begin{document}
\maketitle
\noindent Let $X$ be a (smooth, complex) Fano $4$-fold. As usual, we denote by $\N(X)$ 
the vector space of one-cycles in $X$, with real coefficients, modulo numerical equivalence; its dimension is the Picard number $\rho_X$ of $X$, which coincides with the second Betti number.

For any prime divisor $D\subset X$, let $\N(D,X)$ be the linear subspace of $\N(X)$ generated by classes of curves contained in $D$;  its dimension is at most $\rho_D$. We consider the following invariant of $X$:
$$c_X:=\max\left\{\codim\N(D,X)\,|\,D\text{ a prime divisor in }X\right\}.$$
Notice that $c_X\in\left\{0,\dotsc,\rho_X-1\right\}$, and $c_X\geq \rho_X-\rho_D$ for any prime divisor $D\subset X$. This invariant has been introduced in \cite{codim} for Fano manifolds of arbitrary dimension; it turns out that $c_X$ is always at most $8$ \cite[Th.~3.3]{codim}. 
Here we  consider the case where $X$ has dimension $4$.

If $X=S_1\times S_2$ is a product of Del Pezzo surfaces with \label{product} $\rho_{S_1}\geq\rho_{S_2}$,  then
 $\rho_{S_1}=c_X+1$, 
$\rho_{S_2}\leq c_X+1$, and $\rho_X=\rho_{S_1}+\rho_{S_2}\leq 2c_X+2$ (see \cite[Ex.~3.1]{codim}).
In particular, as Del Pezzo surfaces have Picard number at most $9$, we get 
$c_X\leq 8$ and $\rho_X\leq
18$. 

When $X$ is not a product of surfaces, we have the following.
\begin{thm}[\cite{codim}, Th.~1.1 and Cor.~1.3]\label{tre}
Let $X$ be a Fano $4$-fold which is not a product of surfaces. Then
$c_X\leq 3$.

Moreover if $c_X=3$, then $\rho_X\in\{5,6\}$ and $X$ has a flat fibration onto $\pr^2$, $\pr^1\times\pr^1$, or the Hirzebruch surface $\mathbb{F}_1$. 
\end{thm}
In this paper we consider the case $c_X=2$, in which we give the following bound on the Picard number.
\begin{thm}\label{main}
Let $X$ be a Fano $4$-fold with $c_X=2$. Then $\rho_X\leq 12$.

Moreover if $\rho_X\geq 7$, then there is a diagram:
$$X\longrightarrow X_1\stackrel{h}{\dasharrow}\w{X}_1\longrightarrow Y$$
where all the varieties are smooth and projective,  $X\to X_1$ is the
blow-up of a smooth irreducible surface contained in $\dom(h)$, $X_1$ is Fano,
 $h$ is birational and an isomorphism in
codimension $1$, and
 $\w{X}_1\to Y$ is an elementary contraction and a 
conic bundle.
\end{thm}
The author does not know whether the bound $\rho_X\leq 12$ given above is sharp. Indeed, if $X$ is a product of surfaces with $c_X=2$, then $\rho_X\leq 2c_X+2=6$. On the other hand, all known examples of Fano $4$-folds which are not products of surfaces have Picard number at most $6$; it would be interesting to have examples with larger Picard number.

The main motivation for this work is the following conjecture, which remains open in the case $c_X\leq 1$.
\begin{conjecture}
Let $X$ be a Fano $4$-fold. Then $\rho_X\leq 18$, with equality only if $X$ is a product of surfaces.
\end{conjecture}
 Let us recall that after boundedness, 
the Picard number of Fano manifolds has a maximal value in each dimension. As noticed by J.~Koll\'ar
 \cite[Rem.~1.6]{kollareff},   an explicit bound on $\rho_X$ can be obtained with
the same techniques used to prove boundedness. Indeed
there are explicit bounds for the  Betti numbers of a smooth projective variety, embedded in a projective space, in terms of the dimension and of the degree, as shown recently by
  F.~L.~Zak \cite{zakbounds}.
However such a bound on $\rho_X$ is (conjecturally) far from being sharp: for instance,
 in dimension $4$ we get $\rho_X<3^85^{304}23^426$. We refer the reader to  Rem.~\ref{explicit} for more details.

\medskip

The proof of Theorem \ref{main} relies on both \cite{codim,eff}. 
In particular, we use the following result.
\begin{proposition}[\cite{eff}, Prop.~5.1]\label{51}
Let $X$ be a Fano $4$-fold with $\rho_X\geq 6$ and $c_X=2$.
Then one of the following holds:
\begin{enumerate}[$(i)$]
\item $\rho_X\leq 12$, and there is a diagram
$X\longrightarrow X_1\stackrel{h}{\dasharrow}\w{X}_1\longrightarrow Y$ as in Th.~\ref{main};
 \item there exists a Fano $4$-fold $X_2$ and $\sigma\colon X\to X_2$ a blow-up of
  two disjoint smooth irreducible surfaces.
\end{enumerate}
\end{proposition}
The proof of Prop.~\ref{51} is based on a construction which depends on a prime divisor $D\subset X$ with $\codim\N(D,X)=2$. This construction yields two possible outputs, giving cases $(i)$ and $(ii)$ above. Our strategy to prove Th.~\ref{main} is to exploit the freedom in the choice of $D$: we show that if every prime divisor $D$ with $\codim\N(D,X)=2$ leads to case $(ii)$,  and $\rho_X\geq 7$, then we get a contradiction.

We use the same techniques introduced in \cite{codim}, where
the case $c_X\geq 3$ is studied in arbitrary dimension. The case $c_X=2$ is more difficult, and this is one reason for which we need to work in dimension $4$.

We conclude this introduction by noting that 
the bound $\rho_X\leq 12$ in Th.~\ref{main} follows from the geometric description of the case $\rho_X\geq 7$. Indeed the existence of the rational map $X_1\dasharrow Y$ yields $\rho_{X_1}\leq 11$ by \cite[Th.~1.1]{eff}, and hence $\rho_X\leq 12$.  
The lack of analogous results in higher dimensions is one of the obstructions to study the case $c_X=2$ in general.

\medskip

\noindent\textbf{Acknowledgements.} 
I would like to thank Fyodor L.\ Zak for suggesting a significant improvement in the estimate on $\rho_X$ in Rem.\ \ref{explicit}, in particular 
for pointing out the bound \eqref{zak},
obtained with a more accurate application of his
 results in \cite{zakbounds}.
\subsection*{Notation and terminology}
For any projective variety $X$, we denote by $\N(X)$ (respectively, $\Nu(X)$) the vector space of one-cycles (Cartier divisors), with real coefficients, modulo numerical equivalence. We denote by $[C]$ (respectively, $[D]$) the numerical equivalence class of a curve $C$ (of a Cartier divisor $D$). Moreover, $\NE(X)\subset\N(X)$ is the convex cone generated by classes of effective curves.

\emph{For any closed subset $Z\subset X$, we denote by $\N(Z,X)$ the subspace of $\N(X)$ generated by classes of curves contained in $Z$.}

If $D$ is a Cartier divisor in $X$, we set $D^{\perp}:=\{\gamma\in\N(X)\,|\,D\cdot\gamma=0\}$.

If $X$ is a normal projective variety, a \emph{contraction} of $X$ is a surjective morphism $\ph\colon X\to Y$, with connected fibers, where $Y$ is normal and projective. The contraction is elementary if $\rho_X-\rho_Y=1$.

Let now $X$ be a Fano $4$-fold. Then contractions of $X$ are in bijection with faces of $\NE(X)$, and elementary contractions of $X$ correspond to one-dimensional faces of $\NE(X)$, also called \emph{extremal rays}.

Let $R$ be an extremal ray of $\NE(X)$, and 
 $\ph\colon X \to Y$ the associated contraction. 
If $D$ is a divisor in $X$, the sign of $D\cdot R$ is the sign of $D\cdot C$, $C$ a curve with class in $R$.
We set $\Lo(R):=\Exc(\ph)$, the locus where $\ph$ is not an isomorphism.
We say that $R$ is of type $(a,b)$ where $a=\dim\Exc(\ph)$ and $b=\dim\ph(\Exc(\ph))$. 

\emph{We say that $R$ type $(3,2)^{sm}$ if $Y$ is smooth and $\ph$ is the blow-up of a smooth, irreducible surface in $Y$.} We recall the following very useful property.
\begin{remark}[\cite{wisn}, Th.~1.2]\label{wisn}
Let $X$ be a Fano $4$-fold, $R$ be an extremal ray of $\NE(X)$, and
 $\ph\colon X \to Y$ the associated contraction.  If $\ph$ is birational and  has fibers of dimension $\leq 1$, then $R$ is of type $(3,2)^{sm}$.
In particular, any small extremal ray of $\NE(X)$ is of type $(2,0)$.
\end{remark}
\subsection*{Preliminary results}
In this section we gather some technical results needed for the proof of Th.~\ref{main}.
\begin{remark}\label{easy}
Let $X$ be a Fano manifold, $Z\subset X$ a closed subset, $\ph\colon X\to Y$ a contraction, and $\ph_*\colon\N(X)\to\N(Y)$ the push-forward of one-cycles. Then $\N(\ph(Z),Y)=\ph_*(\N(Z,X))$, so we have:
\stepcounter{thm}
\begin{equation}\label{ea}
\dim\N(\ph(Z),Y)\geq\dim\N(Z,X)-\dim\ker\ph_*.\end{equation}
\end{remark}
\begin{lemma}\label{vari}
  Let $X$ be a Fano $4$-fold with $\rho_X\geq 6$ and $c_X\leq 2$. 
\begin{enumerate}[$(a)$]
\item
Let $E\subset X$ be a prime divisor which is a smooth $\pr^1$-bundle with fiber $f\subset E$, such that $E\cdot f=-1$.
Then $R:=\R_{\geq 0}[f]$ is an extremal ray of type $(3,2)$,
it is 
the unique extremal ray having negative intersection with
$E$, and the target of the contraction of $R$ is Fano.
\item
Suppose that $X$ has
 two extremal rays $R_1$ and $R_2$ of type
$(3,2)$, with loci $E_1$ and $E_2$ respectively, such that $E_1\neq E_2$. 
Then either  $E_1\cdot R_2>0$ and $E_2\cdot R_1>0$, or 
 $E_1\cdot R_2=E_2\cdot R_1=0$. 
\end{enumerate}
\end{lemma}
\begin{proof}
Statement $(a)$ follows from \cite[Rem.~5.3]{eff}.

To show $(b)$, assume by contradiction that  $E_1\cdot R_2>0$ and $E_2\cdot R_1=0$.
The two divisors $E_1$ and $E_2$ intersect along a surface,
because $E_1\cdot R_2>0$. Since $E_2\cdot R_1=0$,  the contraction of $R_1$ sends $E_1\cap E_2$ to a curve, and by \eqref{ea} this yields $\dim\N(E_1\cap E_2,X)\leq 2$.

On the other hand, since $E_1\cdot R_2>0$,
$E_1\cap E_2$ meets every non-trivial fiber of  the contraction $\ph\colon X\to Y$
of $R_2$. This means that $\ph_*(\N(E_2,X))=\ph_*(\N(E_1\cap E_2),X)$, and hence
 $\N(E_2,X)=\N(E_1\cap E_2,X)+\R R_2$.
 Thus
$\dim\N(E_2,X)\leq 1+\dim\N(E_1\cap E_2,X)\leq 3$ and
 $c_X\geq \rho_X-3\geq 3$, a contradiction.
\end{proof}
\begin{lemma}\label{uno}
 Let $X$ be a Fano $4$-fold with $\rho_X\geq 6$ and $c_X=2$.
Let  $D\subset X$ be a prime divisor with $\codim\N(D,X)=2$.
Then there exists an extremal ray 
 $R_1$ of $X$, of type $(3,2)^{sm}$, such that
$D\cdot R_1>0$ and
$R_1\not\subset\N(D,X)$.
\end{lemma}
\begin{proof}
By \cite[Prop.~2.5]{codim}, there exists a prime divisor $E_1\subset X$ which is a smooth $\pr^1$-bundle with fiber $f_1\subset E_1$, such that $E_1\cdot f_1=-1$, $D\cdot f_1>0$, and $[f_1]\not\in\N(D,X)$. By Lemma~\ref{vari}$(a)$,  $R_1:=\R_{\geq 0}[f_1]$ is an extremal
ray of type $(3,2)$. Since $D\cdot R_1>0$, every non-trivial fiber of the contraction of $R_1$ must intersect $D$. On the other hand $R_1\not\subset\N(D,X)$, thus every such fiber must have dimension $1$, and $R_1$ is of type $(3,2)^{sm}$ (see Rem.~\ref{wisn}).
\end{proof}
\begin{lemma}\label{due}
 Let $X$ be a Fano $4$-fold with $\rho_X\geq 6$ and $c_X=2$, and assume that $X$ does not satisfy the statement of Theorem \ref{main}.

Then
for any prime divisor 
  $D\subset X$ with $\codim\N(D,X)=2$, and for any
 extremal ray $R_1$ as in Lemma \ref{uno},
 there exists a second extremal ray $R_2$ of
type $(3,2)^{sm}$, with the following properties, where $E_i:=\Lo(R_i)$ for $i=1,2$:
\begin{enumerate}[$(a)$]
\item  $E_1\cap E_2=\emptyset$;
\item $R_1+R_2$ is a face of $\NE(X)$, whose contraction $\sigma\colon X\to X_2$ is the smooth blow-up of two disjoint irreducible surfaces in $X_2$, and $X_2$ is Fano;
\item  $D\cdot R_i>0$ and
  $R_i\not\subset\N(D,X)$ (in particular $D\neq E_i$), for $i=1,2$;
\item $\codim\N(E_i,X)=2$, $\N(D\cap E_i,X)=\N(D,X)\cap\N(E_i,X)$,
 and $\codim\N(D\cap E_i,X)=3$, for $i=1,2$. 
\end{enumerate}
\end{lemma}
\begin{proof}
We use the construction of a special Mori program for $-D$ introduced in \cite[\S 2]{codim}; we refer the reader to \emph{loc.\ cit}.\ and references therein for the terminology. A special Mori program is just a Mori program where all involved extremal rays have positive intersection with the anticanonical divisor, see \cite[Prop.~2.4]{codim}.

Let $\sigma_0\colon X\to X_1$ be the contraction of $R_1$; notice that  
$X_1$ is Fano by Lemma \ref{vari}$(a)$.
By \cite[Prop.~2.4]{codim}, we can consider a special Mori program for the divisor
$-\sigma_0(D)$ in $X_1$. Together with $\sigma_0$, this yields a special
Mori program for $-D$ in $X$, where the first step is precisely
$\sigma_0$.

We now apply 
\cite[proof of Prop.~5.1]{eff} to this special Mori program. Since we are excluding by assumption case $(i)$ of Prop.~\ref{51}, we know that there exist a smooth Fano $4$-fold $X_2$ and a contraction $\sigma\colon X\to X_2$ which is the 
 blow-up of  two disjoint smooth irreducible surfaces. More precisely,
the proof of  Prop.~\ref{51} shows that $\sigma$ is the contraction of $R_1+R_2$, where
$R_2$ is an extremal ray of type $(3,2)^{sm}$, with locus $E_2$, such that
 $E_1\cap E_2=\emptyset$ and $D\cdot R_2>0$. In particular we have $(a)$ and $(b)$.

Fix $i\in\{1,2\}$. 
By  \cite[Lemma 3.1.8]{codim} we have $\codim\N(E_i,X)=2$ and $\codim\N(D\cap E_i,X)=3$. Notice that $\N(D,X)\neq\N(E_i,X)$, for instance because $\N(E_i,X)$ is contained in $(E_{3-i})^{\perp}$, while $\N(D,X)$ is not.
On the other hand 
$\N(D\cap E_i,X)\subseteq \N(D,X)\cap\N(E_i,X)$, and looking at dimensions we see that equality holds. So we have $(d)$.

Finally  $R_2\not\subset\N(D\cap E_2,X)$ by \cite[Rem.~3.1.3(2)]{codim}, hence
 $R_2\not\subset\N(D,X)$ by $(d)$, and we have $(c)$.
\end{proof}
\begin{lemma}\label{giusto}
 Let $X$ be a Fano $4$-fold with $\rho_X\geq 6$ and $c_X=2$, and assume that $X$ does not satisfy the statement of Theorem \ref{main}.

Let $R$ and $R'$ be two extremal rays of $X$ of type $(3,2)^{sm}$, and set
$E:=\Lo(R)$ and $E':=\Lo(R')$. Suppose that
$\codim\N(E,X)=\codim\N(E',X)=2$, $E\cdot R'>0$, $E'\cdot R>0$,
$R\not\subset\N(E',X)$,
and
$R'\not\subset\N(E,X)$.

Let $S$ be an extremal ray different from $R$ and $R'$. If the
contraction of $S$ is not finite on
$E\cup
E'$, then $E\cdot S=E'\cdot S=0$.
\end{lemma}
\begin{proof}
We notice first of all that, since $S\neq R$ and $S\neq R'$, we have
$E\cdot S\geq 0$ and $E'\cdot S\geq 0$ by Lemma \ref{vari}$(a)$.

Let us assume that the contraction
of $S$ is not finite on $E$, and let $C\subset E$ be an irreducible
curve with $[C]\in S$.

By Lemma \ref{due}, applied with $D=E$ and $R_1=R'$, 
there exists an extremal ray $R''$, of type $(3,2)^{sm}$ and with locus
$E''$, such that:
$$E'\cap E''=\emptyset,\quad E\cdot R''>0,\quad\text{and}\quad R''\not\subset\N(E,X).$$ 
We have $E''\cdot R>0$ by Lemma \ref{vari}$(b)$,
moreover $S\neq R''$ because $S=\R_{\geq 0}[C]\subset\N(E,X)$. 
Therefore $E''\cdot S\geq 0$ by Lemma \ref{vari}$(a)$.

Let $f\subset E$ be an irreducible curve with class in $R$. 
Since $E''\cdot R>0$, we know after
\cite{occhetta} (see \cite[Rem.~3.1.3(3)]{codim}) that there exists
an irreducible curve $C_1$, contained in
$E\cap E''$, such that $C\equiv \lambda
f+\mu C_1$, where $\lambda,\mu\in\Q$, and $\mu\geq 0$. 

Thus $\lambda[f]+\mu[C_1]=[C]\in S$, and since $S$ is an extremal ray of
$\NE(X)$ and $[f]\not\in S$ (for $R\neq S$), we must have $\lambda\leq 0$. On the other
hand since $C_1\subset E''$ and $E'\cap E''=\emptyset$, 
we have $E'\cdot C_1=0$, and hence
$$E'\cdot C=\lambda E'\cdot f\leq 0$$
(recall that $E'\cdot f>0$ by assumption). 
This implies that $E'\cdot S=0$, $\lambda=0$, and $[C_1]\in S$.

Therefore we have shown that if the contraction
of $S$ is not finite on $E$, then $E'\cdot S=0$.

Moreover, since $E'\cdot R>0$, $E''\cdot R>0$, and $E''\cdot S\geq 0$,
we can repeat 
 the same argument 
with the roles of $R'$ and $R''$ interchanged, 
to find $C\equiv \lambda_2
f+\mu_2 C_2$, where $\lambda_2,\mu_2\in\Q$, $\mu_2\geq 0$, and $C_2$ is an irreducible curve contained in $E\cap E'$. In the same way we conclude that 
$[C_2]\in
S$. This means that the contraction of $S$ is not finite on $E'$
neither, thus $E\cdot S=0$ by what precedes. 
\end{proof}
\subsection*{Proof of Theorem \ref{main}}
Let us assume that $\rho_X\geq 7$. Then $X$ cannot be
 a product of surfaces, for otherwise $\rho_X\leq 2c_X+2=6$ (see on p.~\pageref{product}). 

\emph{We proceed by contradiction, and suppose that 
$X$ does not satisfy the statement.}
\begin{step}\label{s1}
There exist three extremal rays $R_0,R_1,R_2$, of type $(3,2)^{sm}$, such that if $E_i:=\Lo(R_i)$ for $i=0,1,2$, we have the following properties:
\begin{enumerate}[$(a)$]
\item
$\codim\N(E_i,X)=2$ for $i=0,1,2$;
\item
the divisors $E_0,E_1,E_2$ are distinct, and $E_1\cap E_2=\emptyset$;
\item $R_1+R_2$ is a face of $\NE(X)$, whose contraction $\sigma\colon X\to X_2$ is the smooth blow-up of two disjoint irreducible surfaces in $X_2$, and $X_2$ is Fano;
\item
$E_0\cdot R_i>0$ and $E_i\cdot R_0>0$ for $i=1,2$;
\item
$R_0\not\subset\N(E_i,X)$ and $R_i\not\subset\N(E_0,X)$ for $i=1,2$;
\item
$\codim\N(E_0\cap E_i)=3$ for $i=1,2$.
\end{enumerate}
\end{step}
\begin{proof}[Proof of step \ref{s1}]
Since $c_X=2$, there exists some prime divisor $D\subset X$ with $\codim\N(D,X)=2$. Lemma \ref{uno} yields the existence of an extremal ray $R_0$, of type $(3,2)^{sm}$, such that $D\cdot R_0>0$ and $R_0\not\subset\N(D,X)$. Then
 Lemma \ref{due}$(d)$ implies, in particular,
that $\codim\N(E_0,X)=2$, where $E_0:=\Lo(R_0)$.

Now we apply Lemmas \ref{uno} and \ref{due} to the divisor $E_0$. We deduce the existence for two extremal rays $R_1$ and $R_2$, of type $(3,2)^{sm}$, with loci $E_1$ and $E_2$ respectively, such that $(a)$, $(b)$, $(c)$, and $(f)$ hold, and moreover $E_0\cdot R_i>0$, 
  $R_i\not\subset\N(E_0,X)$, and
$\N(E_0\cap E_i,X)=\N(E_0,X)\cap\N(E_i,X)$,  for $i=1,2$.

Fix $i\in\{1,2\}$. 
We have $E_i\cdot R_0>0$ by Lemma \ref{vari}$(b)$, hence $(d)$ holds. 
Moreover \cite[Rem.~3.1.3(2)]{codim} yields $R_0\not\subset \N(E_0\cap E_i,X)=\N(E_0,X)\cap\N(E_i,X)$, therefore $R_0\not\subset\N(E_i,X)$, and $(e)$ holds too.
\end{proof}
Consider the blow-up $\sigma\colon X\to X_2$ given by step \ref{s1}$(c)$.
Let
$R'$ be an extremal ray of $\NE(X_2)$ such that $\sigma(E_0)\cdot
R'>0$, and let $\ph\colon X_2\to Y$ be the associated contraction.
We have
 $\rho_Y=\rho_X-3\geq 4$, so that $\dim Y\geq 2$.
$$X\stackrel{\sigma}{\longrightarrow} X_2\stackrel{\ph}{\longrightarrow} Y$$
\begin{step} The case where $\ph$ 
is birational. \end{step}
\noindent We
have $R'=\sigma_*(R_3)$, $R_3$ an extremal ray of $\NE(X)$ 
(see \cite[\S 2.5]{fanos}); in particular $R_3\neq R_1$ and $R_3\neq R_2$.
Since $\sigma(\Lo(R_3))\subseteq \Lo(R')\subsetneq X_2$, the contraction of $R_3$ is birational.
Notice also that $R_3\neq R_0$, for otherwise the locus of $R'$ should be $\sigma(E_0)$, which is excluded because $\sigma(E_0)\cdot R'>0$. By Lemma \ref{vari}$(a)$ we know that $E_i\cdot R_3\geq 0$ for $i=0,1,2$.

We have
$$\sigma^*(\sigma(E_0))=E_0+m_1E_1+m_2E_2$$
with $m_i>0$ (since $E_0\cdot R_i>0$ by step \ref{s1}$(d)$) for $i=1,2$, and $\sigma^*(\sigma(E_0))\cdot R_3>0$ by the projection formula, 
therefore at least one of the intersections $E_0\cdot
R_3$, $E_1\cdot R_3$, $E_2\cdot R_3$  is positive. 

Up to exchanging $R_1$ and $R_2$, we can assume that the intersections $E_0\cdot
R_3$ and $E_1\cdot R_3$ are not both zero. Applying Lemma \ref{giusto}
with $R=R_0$, $R'=R_1$, and $S=R_3$ (notice that the assumptions are satisfied by step \ref{s1}$(a)$,$(d)$,$(e)$), we see that the contraction of
$R_3$ must be finite on $E_0\cup E_1$. On the other hand every curve
with class in $R_3$ must intersect $E_0\cup E_1$, because $(E_0+E_1)\cdot R_3>0$. We conclude that the
contraction of $R_3$ has fibers of dimension at most $1$, therefore
 $R_3$ is of type $(3,2)^{sm}$ (see Rem.~\ref{wisn}). Set
$E_3:=\Lo(R_3)$.

As the extremal rays $R_0,R_1,R_2,R_3$ are distinct, by Lemma~\ref{vari}$(a)$ we know that also the divisors
$E_0, E_1,E_2,E_3$ are distinct, and that $E_3\cdot R_1\geq 0$ and $E_3\cdot R_2\geq 0$.

Suppose that
 $E_1\cdot R_3>0$ and $E_3\cdot R_1>0$, and let $f_3\subset X$ be a curve with class in $R_3$. Using the projection formula one easily sees that $\sigma(E_3)\cdot \sigma_*(f_3)\geq 0$. On the other hand $[\sigma_*(f_3)]\in R'$, and $R'$ is divisorial with locus $\sigma(E_3)$, so we have a contradiction.

We conclude by Lemma~\ref{vari}$(b)$ that
$E_1\cdot R_3=E_3\cdot R_1=0$. Moreover, since we are assuming that the intersections $E_0\cdot
R_3$ and $E_1\cdot R_3$ are not both zero, we must have
$E_0\cdot R_3>0$. 

If $E_1\cap E_3\neq\emptyset$, then $E_1$ should contain some curve with
class in $R_3$, which is impossible because the contraction of $R_3$
is finite on $E_1$. Thus $E_1\cap E_3=\emptyset$.

Similarly, applying Lemma \ref{giusto} with $R=R_0$, $R'=R_2$, and
$S=R_3$ (again, the assumptions are satisfied by step \ref{s1}), we conclude that the contraction of $R_3$ is finite on $E_2$
too, that $E_2\cdot R_3=E_3\cdot R_2=0$, and finally that $E_2\cap
E_3=\emptyset$.

Therefore $E_1,E_2,E_3$ are pairwise disjoint $\pr^1$-bundles, with
fibers $f_i\subset E_i$, such that $E_i\cdot f_i=-1$ and $E_0\cdot
f_i>0$, for $i=1,2,3$. Now \cite[Lemma 3.1.7]{codim} yields that $\N(E_0\cap E_1,X)\subseteq
E_i^{\perp}$ for $i=1,2,3$.

Applying Lemma \ref{due} to the divisor $E_1$ and the extremal ray
$R_0$, we find
 a prime divisor $E_4$ such that $E_4\neq E_1$, $E_0\cap E_4=\emptyset$, and
$E_1\cap E_4\neq\emptyset$. Since $E_2$ and $E_3$ are disjoint from
 $E_1$, we also have $E_4\neq E_2$ and $E_4\neq E_3$, so that $E_4\cdot
 f_i\geq 0$ for $i=1,2,3$. 

For every curve $C\subset E_0\cap E_1$ we have $C\cap E_4=\emptyset$,
therefore $E_4\cdot C=0$; this gives
 $\N(E_0\cap E_1,X)\subseteq E_4^{\perp}$.

Summing-up, we have  $\N(E_0\cap E_1,X)\subseteq E_1^{\perp}\cap E_2^{\perp}\cap E_3^{\perp}\cap E_4^{\perp}$
Since $\codim\N(E_0\cap E_1,X)=3$ by step \ref{s1}$(f)$, we deduce that $[E_1],[E_2],
[E_3],[E_4]$ are linearly dependent in $\Nu(X)$, thus there exist
rational numbers $a_i$, not all zero, such that $\sum_{i=1}^4a_iE_i\equiv 0$.
Fix $j\in\{1,2,3\}$. Intersecting with $f_j$ we get $a_j=a_4E_4\cdot
f_j$. In particular $a_4\neq 0$, and we get:
$$(E_4\cdot f_1)E_1+(E_4\cdot f_2)E_2+(E_4\cdot f_3)E_3+E_4\equiv 0.$$
Since a non-zero effective divisor cannot be numerically trivial, we
have a contradiction.
\begin{step}
The case where $\ph$ is of fiber type.
\end{step} 
\noindent Suppose first that $Y$ is a surface. 

Recall that
$\sigma(E_1)$ and
$\sigma(E_2)$ are the two surfaces blown-up by $\sigma$. Fix $i\in\{1,2\}$.
By \eqref{ea} we have
  $\dim\N(\ph(\sigma(E_i)),Y)\geq\dim\N(E_i,X)-3=\rho_X-5\geq 2$.
 Therefore $\ph(\sigma(E_i))$ cannot be a point nor a curve, and
we conclude that $\sigma(E_i)$ dominates $Y$ under $\ph$.

Consider now a general fiber $F\subset X_2$ of $\ph$. Then $\dim\N(F,X_2)=1$,
and $F$ intersects both
$\sigma(E_1)$ and
$\sigma(E_2)$. The inverse image  $\sigma^{-1}(F)$ is a general fiber of the composition $\ph\circ\sigma\colon X\to Y$,
and it
contains curves
with class in $R_1$ and in $R_2$, so that 
$\dim\N(\sigma^{-1}(F),X)=3=\dim\ker(\ph\circ\sigma)_*$. 
This means that $\ph\circ\sigma$ is a ``quasi-elementary'' contraction (see
\cite[Def.~3.1]{fanos}), 
 and by \cite[Th.~1.1]{fanos} $X$
is a product of surfaces, a contradiction.

\bigskip

Suppose now that $\dim Y=3$, so that $Y$ is factorial with isolated
canonical singularities (this is well-known -- see for instance \cite[Lemma 3.10(i)]{fanos} and references therein). Since $\sigma(E_0)\cdot R'>0$, the divisor $\sigma(E_0)$ intersects every fiber of $\ph$, 
and we have $(\ph\circ\sigma)(E_0)=Y$. Hence 
 $\ph\circ\sigma$ is generically finite on $E_0$, and cannot contract
the fibers $f_0$ of the $\pr^1$-bundle structure on $E_0$ given by the contraction of $R_0$. The images
$(\ph\circ\sigma)(f_0)$ give a 
covering and unsplit family of rational curves in $Y$ (see \cite{unsplit} for the terminology), and by
\cite[Cor.~1]{unsplit} there exists an elementary contraction 
$\psi\colon Y\to Z$ which contracts the curves $(\ph\circ\sigma)(f_0)$. In
particular $\psi$ is of fiber type, and $\rho_Z=\rho_Y-1\geq 3$, hence
$Z$ is a surface. 
$$X\stackrel{\sigma}{\longrightarrow} X_2\stackrel{\ph}{\longrightarrow} Y
\stackrel{\psi}{\longrightarrow} Z$$

Let $F\subset X_2$ be a general fiber of $\psi\circ\ph\colon X_2\to
Z$. Then $\dim\N(F,X_2)=2$, and $F$ contains some curve of the type
$\sigma(f_0)\subset\sigma(E_0)$. Notice that the surfaces $\sigma(E_1)$ and
$\sigma(E_2)$ 
both intersect $\sigma(f_0)$, 
because $E_1\cdot R_0>0$ and $E_2\cdot
R_0>0$ by step \ref{s1}$(d)$.

The inverse image $\sigma^{-1}(F)$ is a general fiber of the
 composition $\xi:=\psi\circ\ph\circ\sigma\colon X\to
Z$. Since $\sigma(E_1)$ and
$\sigma(E_2)$  both intersect $F$,  $\sigma^{-1}(F)$ contains curves
with class in $R_1$ and in $R_2$, so that 
$\N(\sigma^{-1}(F),X)\supseteq \ker\sigma_*$. As $\sigma_*(\N(\sigma^{-1}(F),X))=\N(F,X_2)$,
we get
$\dim\N(\sigma^{-1}(F),X)=2+\dim\N(F,X_2)=4$.
On the other hand
$\dim\ker\xi_*=4$, and as in the previous case,
this means that $\xi$ is a ``quasi-elementary'' contraction. 
 Therefore by \cite[Th.~1.1]{fanos} $X$
is a product of surfaces, which is again a contradiction.
This concludes the proof of Th.~\ref{main}.
\qed
\begin{remark}\label{explicit} 
 Let $X$ be a Fano manifold of dimension $n$, and assume that $X$ is embedded in a projective space with degree $d$. Then it follows from \cite{zakbounds} that
\stepcounter{thm}
\begin{equation}\label{zak}
\rho_X<\left(n^2+n+2\right)d.\end{equation}
Indeed let us consider the three {classes} $\mu_0$, $\mu_1$, and $\mu_2$ of $X$. These are projective invariants of an embedded variety (see \cite{zakbounds} for the definition), in particular $\mu_0=d$ and $\mu_1=2d+2g-2$, where $g$ is the sectional genus of $X$. Using the fact that $X$ is Fano,
one can easily check that $2g-2<(n-1)d$ and hence $\mu_1<(n+1)d$.
Now it follows from \cite[Cor.~1.13 and Th.~2.9(ii)]{zakbounds} that  
$$\rho_X\leq \mu_2+\mu_0\leq \frac{\mu_1^2}{d}+d<\left(n^2+2n+2\right)d.$$

On the other hand there are explicit constants $\delta_n$ and $m_n$, depending only on $n$, such that
$-m_nK_X$ is very ample, and $(-K_X)^n\leq \delta_n$ (see \cite[\S 5.9]{debarreUT} and references therein). Therefore $X$ can always be embedded  with degree at most $m_n^n\delta_n$, and
\eqref{zak} yields $\rho_X<(n^2+2n+2)m_n^n\delta_n$. 

In dimension $4$ we can take $\delta_4=3^85^{304}$ \cite[Th.~5.18]{debarreUT}. 
Moreover 
 $-4K_X$ is free \cite{kawfreeness}, hence $-23K_X$ is very ample, see \cite[Ex.~1.8.23]{lazI}. This yields 
$$\rho_X<3^85^{304}23^426$$
for any Fano $4$-fold $X$.
 Notice that even if we take $m_4=5$ (as predicted by Fujita's conjecture) and $\delta_4=800$ (which, to the author's knowledge, is the maximal anticanonical degree among the known examples of Fano $4$-folds), we still get a rather large bound.
 \end{remark}

\footnotesize
\providecommand{\bysame}{\leavevmode\hbox to3em{\hrulefill}\thinspace}
\providecommand{\MR}{\relax\ifhmode\unskip\space\fi MR }
\providecommand{\MRhref}[2]{%
  \href{http://www.ams.org/mathscinet-getitem?mr=#1}{#2}
}
\providecommand{\href}[2]{#2}

\bigskip

\bigskip

\noindent C.\ Casagrande\\
 Dipartimento di Matematica, Universit\`a di Torino \\
via Carlo Alberto, 10 \\
 10123 Torino - Italy \\
cinzia.casagrande@unito.it
\end{document}